\documentclass[11pt, a4paper]{article}
\usepackage{amsmath,amssymb,amsfonts,graphicx,amsthm}

\usepackage{amsthm}
\usepackage[utf8]{inputenc}
\usepackage{graphicx}
\usepackage{authblk}
\usepackage[T1]{fontenc}

\newtheorem{thm}{{\bf  Theorem}}
\newtheorem{cor}{{\bf  Corollary}}
\newtheorem{preconj}{{\bf  Conjecture}}
\input epsf

\begin{document}

\title{\bf Double-Star Decomposition of Regular Graphs
\thanks {2010{ \it Mathematics Subject Classification}: 05C51, 05C05.
 }}

\author{{\normalsize
{\sc S. Akbari${}^{\mathsf{a}, \mathsf{c}}$},\,
  {\sc Sh. Haghi${}^{\mathsf{b}, \mathsf{c}}$},\,
  {\sc H.R. Maimani${}^{\mathsf{b}, \mathsf{c}}$}\,
  {\sc and A. Seify${}^{\mathsf{b}, \mathsf{c}}$}\,}
 \vspace{3mm}
\\{\footnotesize{${}^{\mathsf{a}}$\it Department of
Mathematical Sciences, Sharif University of Technology, Tehran,
Iran}}
{\footnotesize{}}\\{\footnotesize{${}^{\mathsf{b}}$\it Mathematics Section, Department of Basic Sciences, Shahid Rajaee Teacher Training University, P.O. Box 16783-163,  Tehran, Iran}}
{\footnotesize{}}\\{\footnotesize{${}^{\mathsf{c}}$\it School of Mathematics, Institute for Research in Fundamental Sciences (IPM),}}{\footnotesize{}}\\{\footnotesize{${}^{\mathsf{}}$\it
P.O. Box 19395-5746,
 Tehran, Iran.}}
\thanks{{\it E-mail addresses}: $\mathsf{s\_akbari@sharif.edu}$,
$\mathsf{sh.haghi@ipm.ir}$, $\mathsf{maimani@ipm.ir}$ and $\mathsf{abbas.seify@gmail.com}$.} }

\date{}

\maketitle

\begin{abstract}
A tree containing exactly two non-pendant vertices is called a double-star. A double-star with degree sequence $(k_1+1, k_2+1, 1, \ldots, 1)$ is denoted by $S_{k_1, k_2}$. We study the edge-decomposition of regular graphs into double-stars. It was proved that every double-star of size $k$ decomposes every $2k$-regular graph. In this paper, we extend this result to $(2k+1)$-regular graphs, by showing that every $(2k+1)$-regular graph containing two disjoint perfect matchings is decomposed into $S_{k_1, k_2}$ and $S_{k_1-1, k_2}$, for all positive integers $k_1$ and $k_2$ such that $k_1+k_2=k$.  
\end{abstract}

\section{Introduction}
Let $G=(V(G),E(G))$ be a graph and $v \in V(G)$. We denote the set of all neighbors of $v$ by $N(v)$. The degree of a vertex $v$ in $G$ is denoted by $d_G(v)$ (for abbreviation $d(v)$). By {\it size} and {\it order} of $G$ we mean $|E(G)|$ and $|V(G)|$, respectively. Let $X \subseteq V(G)$, then the {\it induced subgraph} with vertex set $X$ is denoted by $G[X]$. A subset $M \subseteq E(G)$ is called a {\it matching} if no two edges of $M$ are incident. A matching $M$ is called a {\it perfect matching}, if every vertex of $G$ is incident with an edge of $M$.
A {\it factor} of $G$ is a spanning subgraph of $G$. A subgraph $H$ is called an $r$-{\it factor} if $H$ is a factor of $G$ and $d_H(v)=r$, for every $v \in V(G)$. 
\\
If $d(v)=1$, then $v$ is called a {\it pendant vertex}. A tree containing exactly two non-pendant vertices is called a {\it double-star}. A double-star with degree sequence $(k_1+1, k_2+1, 1, \ldots, 1)$ is denoted by $S_{k_1, k_2}$. Suppose that $u_1, u_2 \in V(S_{k_1, k_2})$ and $d(u_i)=k_i+1$, for $i=1,2$. Let $X$ and $Y$ be the set of all pendant vertices adjacent to $u_1$ and $u_2$, respectively. Then we say that $S_{k_1, k_2}$ is a double-star with pendant sets $X$ and $Y$. Also, $e=u_1u_2$ is called the {\it central edge} of the double-star.
\\
Let $G$ and $H$ be two graphs. The {\it Cartesian product} of $G$ and $H$ is denoted by $G\, \square \,H$ and is a graph with vertex set $\{(u, v): u\in G , v \in H\}$
and two vertices $(u, v)$ and $(x, y)$ are adjacent if and only if $u = x$ and $v$ is adjacent to $y$, or  $v = y$ and $u$ is adjacent to $x$.
\\
For a graph $H$ the graph $G$ has an $H$-{\it decomposition}, if all edges of $G$ can be partitioned into subgraphs isomorphic to $H$. Also, we say that $G$ has an $\{H_1,\ldots,H_t\}$-decomposition if all edges of $G$ can be partitioned into subgraphs, each of them isomorphic to some $H_i$, for $1 \leq i \leq t$. If $G$ has an $H$-decomposition, we say that $G$ is $H$-{\it decomposable}. A graph $G$ is $k$-{\it factorable} if it can be decomposed into $k$-factors. 
\\
Let $G$ be a directed graph and $v \in V(G)$. We define $N^{+}(v)=\{u \in V(G): (v,u) \in E(G)\}$, where $(v,u)$ denotes the edge from $v$ to $u$. By {\it out-degree} of $v$ we mean $|N^{+}(v)|$ and denote it by $d_G^{+}(v)$. Similarly, we define $N^{-}(v)=\{u \in V(G): (u,v) \in E(G)\}$ and denote $|N^-(v)|$ by $d^-_G(v)$. An orientation $O$ is called {\it Eulerian} if $d^{+}_G(v)=d^{-}_G(v)$, for every $v \in V(G)$. A $k$-{\it orientation} is an orientation such that $d_G^{+}(v)=k$, for every $v \in V(G)$. A $\{k_1, \ldots, k_t\}$-{\it orientation} is an orientation such that for every $v \in V(G)$, $d_G^{+}(v)=k_i$, for some $1 \leq i \leq t$. 
\\
In 1979, K\"{o}tzig conjectured that every $(2k+1)$-regular graph is decomposed into $S_{k, k}$ if and only if it has a perfect matching. Jaeger, Payan and Kouider in 1983 proved this conjecture, see \cite{Jaeger}. El-Zanati et al. proved that every $2k$-regular graph containing a perfect matching is $S_{k, k-1}$-decomposable, see \cite{double star}. The following interesting conjecture was proposed by Ringel, see \cite{Ringel}. 

\begin{preconj}
Every tree of size $k$ decomposes the complete graph $K_{2k+1}$.
\end{preconj}

A broadening of Ringel's conjecture is due to Graham and H\"{a}ggkvist.

\begin{preconj}
Every tree of size $k$ decomposes every $2k$-regular graph.
\end{preconj}

El-Zanati et al. proved the following theorem in \cite{double star}.

\begin{thm} \label{double-n2n}
Every double-star of size $k$ decomposes every $2k$-regular graph.
\end{thm}

Jacobson et al. in 1991 proposed the following conjecture about the tree decomposition of regular bipartite graphs, see \cite{bipartite}.

\begin{preconj}\label{conj}
Let $T$ be a tree of size $r$. Then every $r$-regular bipartite graph is $T$-decomposable.
\end{preconj}

They proved that the conjecture holds for double-stars. In this paper, we study double-star decomposition of regular graphs. First, we prove some results about the double-star decomposition of regular bipartite graphs. We present a short proof for Conjecture \ref{conj}, when $T$ is a double-star. Then we study the double-star decomposition of $(2k+1)$-regular graphs. As a straight result of Theorem \ref{double-n2n}, we prove that every $(2k+1)$-regular graph containing a $2$-factor, has $S$-decomposition, where $S=\{S_{k_1, k_2}, S_{k_1-1, k_2}, S_{k_1, k_2-1}, S_{k_1-1, k_2-1}\}$ for any double-star $S_{k_1, k_2}$ of size $k+1$. Then we present a theorem which indicates that every $(2k+1)$-regular graph containing two disjoint perfect matchings is $\{S_{k_1, k_2}, S_{k_1-1, k_2}\}$-decomposable, for any double-star $S_{k_1, k_2}$ of size $k+1$. Also, we prove that every triangle-free $(2k+1)$-regular graph containing a $2$-factor, is $\{S_{k_1, k_2}, S_{k_1+1, k_2}\}$-decomposable, for any double-star $S_{k_1, k_2}$ of size $k$.

\section{Double-Star Decomposition of Regular Bipartite Graphs} 
In this section, we prove some results about the double-star decomposition of regular bipartite graphs. The following theorem was proved in \cite{bipartite}. We present a short proof for this result.

\begin{thm}\label{double-bipartite}
For $r \geq 3$, let $G$ be an $r$-regular bipartite graph. Then every double-star of size $r$ decomposes $G$.
\end{thm} 

\begin{proof}
Let $A$ and $B$ be two parts of $G$. Then K\"{o}nig Theorem \cite[Theorem 2.2]{Kano} implies that $G$ has a 1-factorization with 1-factors $M_1, \ldots, M_r$. Suppose that $S_{k_1,k_2}$ is a double-star of size $r$. Now, let $G_1$ and $G_2$ be two induced subgraphs of $G$ with the edges $M_1 \cup M_2 \cup \cdots \cup M_{k_1}$ and $M_{k_1+1} \cup \cdots \cup M_{r-1}$, respectively. Suppose that $e=u_1u_2 \in M_r$, where $u_1 \in A$ and $u_2 \in B$. Now, define $S_e$ the double-star containing the central edge $e$, $E_1(u_1)$ and $E_2(u_2)$, where $E_i(u_i)$ is the set of all edges incident with $u_i$ in $G_i$. Clearly, $S_e$ is isomorphic to $S_{k_1,k_2}$. On the other hand, $S_{e}$ and $S_{e'}$ are edge disjoint, for every two distinct edges $e, e' \in M_r$. Hence, $E(G)=\cup_{e \in M_1} S_e$ and this completes the proof.
\end{proof}

Now, we have the following corollaries.
\begin{cor}
Let $r,s \geq 3$ be positive integers and $s\,|\,r$. Then every $r$-regular bipartite graph can be decomposed into any double-star of size $s$. 
\end{cor} 

\begin{proof}
Let $r=sk$ and $S_{k_1,k_2}$ be a double-star of size $s$. Since $G$ is 1-factorable, $G$ can be decomposed into spanning subgraphs $G_1, \ldots, G_k$, each of them is $s$-regular. Now, Theorem \ref{double-bipartite} implies that each $G_i$ can be decomposed into $S_{k_1,k_2}$ and this completes the proof. 
\end{proof}

\begin{cor}
Let $r, s, k$ and $t$ be positive integers such that $r=sk+t$ and $r, s, t \geq 3$. Moreover, suppose that $S_{1}$ and $S_{2}$ be two double-stars of size $s$ and $t$, respectively. Then every $r$-regular bipartite graph $G$ is $\{S_{1}, S_{2}\}$-decomposable. 
\end{cor} 

\begin{proof}
Similar to the proof of the previous corollary, $G$ is decomposed into $G_1, \ldots, G_{k+1}$ where $G_1, \ldots, G_k$ are $s$-regular and $G_{k+1}$ is $t$-regular. Now, Theorem \ref{double-bipartite} implies that $G_1, \ldots, G_k$ and $G_{k+1}$ can be decomposed into $S_{1}$ and $S_{2}$, respectively. This completes the proof.
\end{proof}

Another generalization of Theorem \ref{double-bipartite} is as follows.

\begin{thm}\label{gen-bi}
Let $r \geq 3$ be an integer and $G=(A,B)$ be a bipartite graph such that for every $v \in V(G)$, $r\,|\,d(v)$. Then every double-star of size $r$ decomposes $G$.
\end{thm}

\begin{proof}
Let $v \in A$ and $d(v)=rk$, for some positive integer $k$ and $S$ be a double-star of size $r$. Suppose that $N(v)=\{ u_1, \ldots, u_{rk} \} $. Let $G'$ be the graph obtained from $G$ by removing $v$ and adding $v_1, \ldots, v_k$ to $A$. For $i=1, \ldots, k$, join $v_i$ to every vertex of the set $ \{ u_{(i-1)r+1}, \ldots, u_{ir} \}  $. It is not hard to see that if $G'$ is $S$-decomposable, then $G$ is $S$-decomposable, too.
\\
By repeating this procedure for all vertices of $G$, one can obtain an $r$-regular bipartite graph, say $H$. Now,
Theorem \ref{double-bipartite} implies that $H$ is $S$-decomposable and hence $G$ is $S$-decomposable.
\end{proof}

\section{Double-Star Decomposition of Odd Regular Graphs}
In this section, we use the following structure which was used in \cite{bipartite}. Let $G$ be a $2k$-regular graph. Then Petersen Theorem \cite[Theorem 3.1]{Kano} implies that $G$ is $2$-factorable. Let $F$ be a $2$-factor of $G$ with cycles $C_1, \ldots, C_l$. Note that $G \setminus F$ has an Eulerian orientation. Also, orient $C_i$ clockwise, for $i=1, \ldots, l$, to obtain an Eulerian orientation of $G$. We define $G_{C_i}$ as the subgraph of $G$ with the edge set $E=\{ (u,v): u \in V(C_i) \}$. Clearly, $\{ G_{C_1}, \ldots, G_{C_l} \}$ partitions $E(G)$. So, if we show that each $G_{C_i}$ is $S_{k_1, k_2}$-decomposable, then $G$ is $S_{k_1, k_2}$-decomposable too, for every double-star $S_{k_1, k_2}$ of size $k$. In \cite[Theorem 3]{double star}, the following was proved.

\begin{thm}\label{saad}
Let $C_i: v_1, e_1, \ldots, v_t, e_t, v_1$, where $1 \leq i \leq l$. Then $N^{+}_{G \setminus F}(v_j)$ can be partitioned into $X_j$ and $Y_{j-1}$ such that $|X_j|=k_1$, $|Y_{j-1}|=k_2$ and $G_{C_i}[\{v_j, v_{j+1}\} \cup X_j \cup Y_j ]$ is isomorphic to $S_{k_1, k_2}$, for $j=1, \ldots, t$.
\end{thm}

As a straightforward result of Theorem \ref{saad}, we prove the following corollary.
\begin{cor}
Let $G$ be a $(2k+1)$-regular graph and $k_1$ and $k_2$ be two positive integers such that $k_1+k_2=k$. If $G$ has a $2$-factor, then $G$ is $S$-decomposable, for $S=\{S_{k_1,k_2},S_{k_1-1,k_2},S_{k_1,k_2-1},S_{k_1-1,k_2-1}\}$. 
\end{cor}

\begin{proof}
Let $H= G \,\square \,K_2$. Now, $H$ is a $(2k+2)$-regular graph. Let $F_1$ and $F_2$ be two 2-factors in two copies of $G$, namely $G_1$ and $G_2$. Clearly, $F_1 \cup F_2$ is a 2-factor in $H$. Now, Theorem \ref{saad} implies that $H$ has an $S_{k_1,k_2}$ decomposition in which the central edges of double stars are exactly the edges of $F_1 \cup F_2$. Clearly, in this decomposition each double-star has at most two edges with one end point in $V(G_1)$ and other end point in $V(G_2)$. By restriction of this decomposition to $G_1$, we are done.
\end{proof}

Here, we prove two results about the double-star decomposition of $(2k+1)$-regular graphs. 

\begin{thm}\label{main}
Let $G$ be a $(2k+1)$-regular graph containing two disjoint perfect matchings $M_1$ and $M_2$. If $k_1$ and $k_2$ are two positive integers such that $k_1 + k_2 = k$, then $G$ is $\{S_{k_1,k_2}, S_{k_1-1, k_2} \}$-decomposable. 
\end{thm}

\begin{proof}
Since $G \setminus M_1$ is a $2k$-regular graph, it has an Eulerian orientation. Consider an orientation for $M_1$. So, $G$ has a $\{k,k+1\}$-orientation, say $O$, where $M_1$ is a perfect matching between vertices of out-degree $k$ and vertices of out-degree $k+1$. Let $H=G \, \square \, K_2$ and $G_1$ and $G_2$ be two copies of $G$ in $H$ with $V(G_1)=\{v_1, \ldots, v_n\}$ and $V(G_2)=\{v'_1, \ldots, v'_n\}$. Also, suppose that $E'$ is the set of all edges between $G_1$ and $G_2$ i.e. $E'=\{v_iv'_i| i=1, \ldots, n\}$. Clearly, $H$ is $(2k+2)$-regular. Consider orientations $O$ and $O'$ for $G_1$ and $G_2$, respectively, where orientation of $O'$ is reverse of $O$. Also, orient the edges of $E'$ from the vertices of out-degree $k$ to the vertices of out-degree $k+1$. This is an Eulerian orientation for $H$. Let $F_1=M_{1} \cup M_{2}$ and $F_2=M'_{1} \cup M'_{2}$ in which $M'_{i}$ is corresponding perfect matching in $G_2$, for $i=1, 2$.
Now, $F=F_1 \cup F_2$ is a 2-factor of $H$ and Theorem \ref{saad} implies that $H$ can be decomposed into $S_{k_1,k_2}$ such that all the central edges of double-stars are the edges of $F$. Now, we prove two following claims.
\vspace{1em}
\\
\textbf{Claim 1.} The graph $H$ has an $S_{k_1,k_2}$-decomposition such that each double-star in $G_1$ has at most one edge in $E'$. 
\vspace{0.5em}
\\
Let $C:v_1, e_1, v_2, \ldots, v_l, e_l, v_1$ be a cycle in $F_1$. Suppose that the double-star $S_{e_i}$, corresponding to $e_i$, has two edges in $E'$. This implies that $d^{+}_{G_1}(v_i)=d^{+}_{G_1}(v_{i+1})=k$ and $e_i\in{M_2}$. Thus $e_{i-1}, e_{i+1}\in{M_1}$ and hence $d^{+}_{G_1}(v_{i-1})=d^{+}_{G_1}(v_{i+2})=k+1$. For every $v_j \in V(C)$, let $X_j$ be the set of all pendant vertices adjacent to $v_j$ in $S_{e_j}$. Similarly, let $Y_j$ be the set of all pendant vertices adjacent to $v_{j+1}$ in $S_{e_j}$. With no loss of generality assume that $|X_j|=k_1$ and $|Y_j|=k_2$, for $j=1, \ldots, l$. Since $S_{e_i}$ has two edges in $E'$, we have $v'_i \in X_i$ and $v'_{i+1} \in Y_i$. We have $v'_i \notin X_{i-1}$ and $v'_{i+1} \notin Y_{i+1}$.
\\
Note that since $|X_i|=|X_{i+1}|=k_1$ and $v'_{i} \notin X_{i+1}$, we conclude that there exists $t\in X_{i+1} \setminus X_i$. Now, define $X'_{i+1}=(X_{i+1} \setminus \{t\}) \cup \{v'_{i+1}\}$ and $Y'_{i}=(Y_{i} \setminus \{v'_{i+1}\}) \cup \{t\}$, see Figure 1.  Let $S'_{e_{i+1}}$ be the double-star with pendant sets $X'_{i+1}$ and $Y_{i+1}$. Similarly, let $S'_{e_i}$ be the double-star with pendant sets $X_{i}$ and $Y'_i$. For any $j \neq i,i+1$, let $S'_{e_j}=S_{e_j}$. 

\begin{figure}
\centering{\includegraphics[width=105mm]{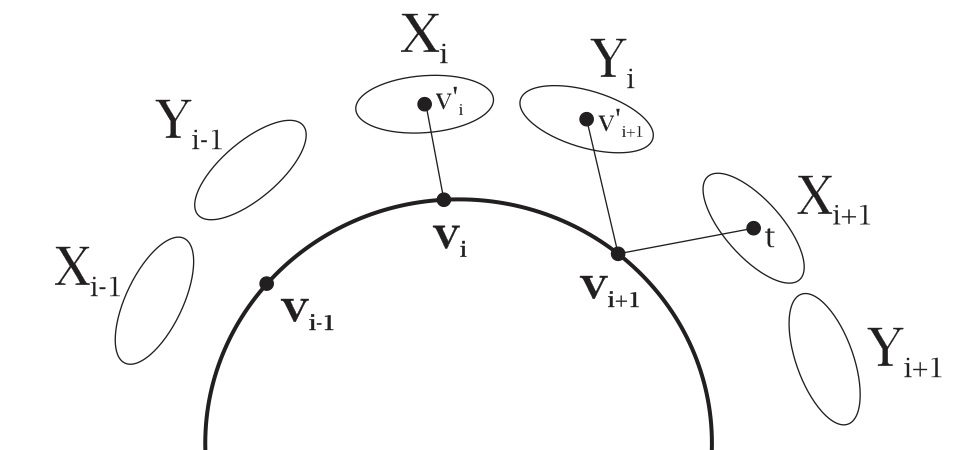}}
\\
{Figure 1.} 
\end{figure}
By repeating this procedure, one can obtain an $S_{ k_1, k_2 }$-decomposition for $H$ in which every double-star has at most one edge in $E'$. This completes the proof of Claim 1.
\vspace{1em}
\\
\textbf{Claim 2.} There exists an $S_{k_1, k_2}$-decomposition for $H$ in which every edge of each double-star contained in $E'$ is incident with a vertex of degree $k_1$.
\vspace{0.5em}
\\
Consider the decomposition given in the Claim 1. Suppose that $v'_{i+1} \in Y_i$. With no loss of generality, we may assume that $v'_2 \in Y_1$. Thus, $v'_2 \notin Y_2 \cup \, Y_l$. This implies that $Y_1 \notin \{Y_2, Y_l\}$. Now, we consider two cases:
\vspace{1em}
\\
\textbf{Case 1.} $X_1 \neq X_2$.
\vspace{0.5em}
\\
There exists $x \in X_2 \setminus X_1$. Now, define $Y'_1=(Y_1 \setminus \{v'_2\}) \cup \{x\}$ and $X'_{2}=(X_{2} \setminus \{x\}) \cup \{v'_2\}$. Let $S'_{e_1}$ be the double-star with pendant sets $X_1$ and $Y'_1$ and $S'_{e_2}$ be the double-star with pendant sets $X'_2$ and $Y_2$. Now, by considering $S'_{e_j}= S_{e_j}$, for $j\in \{3, 4, \ldots, l\}$, the result follows.
\vspace{1em}
\\
\textbf{Case 2.} $X_1=X_2$. We consider two subcases:
\vspace{0.5em}
\\
\textbf{Subcase 2.1.} $X_1=X_2= \cdots = X_l$.
\vspace{0.5em}
\\
\begin{figure}
\centering{\includegraphics[width=65mm]{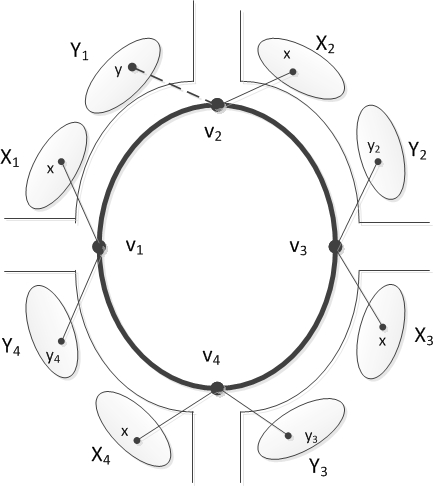}} 
\\
{Figure 2.}
\end{figure}
Let $x \in X_1$. Since $Y_1 \neq Y_l$, there exists $y_l \in Y_l \setminus Y_1$. Now, let $X'_1= (X_1 \setminus \{x\}) \cup \{y_{l}\}$ and $Y'_l=(Y_l \setminus \{y_l\}) \cup \{x\}$. Since $X_1=X_l$ and $x \in Y'_l$, we conclude that $x \notin Y_{l-1}$ and hence there exists $y_{l-1} \in Y_{l-1} \setminus Y'_l$. Now, let $X'_l= (X_l \setminus \{x\}) \cup \{y_{l-1}\}$ and $Y'_{l-1}=(Y_{l-1} \setminus \{y_{l-1}\}) \cup \{x\}$. Repeat this procedure $l$ times. In the final step, let $X'_2= (X_2 \setminus \{x\}) \cup \{v'_2\}$ and $Y'_1=(Y_1 \setminus \{v'_2\}) \cup \{x\}$. One can see that $X'_i \cap Y'_i = \emptyset$, for $i=1,2, \ldots, l$. Let $S'_{e_i}$ be the double-star with pendant sets $X'_i$ and $Y'_i$, see Figure 2. By repeating this procedure we obtain the desired decomposition.
\vspace{1em}
\\
\textbf{Subcase 2.2.} $X_1\neq X_t$, for some $3 \leq t \leq l$.
\vspace{0.5em}
\\
Without loss of generality, we may assume that $X_1=X_l= \cdots =X_{t+1} \neq X_t$. Hence there exists $x \in X_1 \setminus X_t$. Similar to the previous subcase, we can define $X'_1, X'_l, \ldots, X'_{t+1}$ and $Y'_l, Y'_{l-1}, \ldots, Y'_t$. We have $X'_1 \neq X_2$ and according to the Case 1, we are done. 
\vspace{1em}
\\
Now, these two claims yield that the restriction of decomposition of $H$ to $G_1$ is an $\{S_{k_1, k_2}, S_{k_1-1, k_2}\}$-decomposition. This completes the proof.
\end{proof}

In the following theorem, we find another result about the double-star decomposition of $(2k+1)$-regular graphs.

\begin{thm}
Let $G$ be a $(2k+1)$-regular graph containing a perfect matching $M$ and $F$ be a $2$-factor in $G\setminus M$. Suppose that $k_1$ and $k_2$ are two positive integers such that $k_1 + k_2 = k-1$. If for every two adjacent vertices $u$ and $v$ in $F$, $|N_G(u) \cap N_G(v)| \leq k_1-1$, then $G$ is $\{ S_{k_1, k_2}, S_{k_1+1, k_2} \}$-decomposable.
\end{thm}

\begin{proof}
Theorem \ref{saad} implies that $G \setminus M$ can be decomposed into $S_{k_1,k_2}$ such that the edges of $F$ are the central edges of double-stars. Now, let $e=uv \in M$ and $C:v_1, e_1, v_2, \ldots, v_t, e_t, v_1$ be a cycle in $F$ containing $u$. Without loss of generality we may assume that $u=v_1$. Suppose that $X_i$ is the set of all pendant vertices adjacent to $v_i$ in $S_{e_i}$, for $i=1,2, \ldots, t$. Similarly, let $Y_i$ be the set of all pendant vertices adjacent to $v_{i+1}$ in $S_{e_i}$. With no loss of generality assume that $|X_i|=k_1$ and $|Y_i|=k_2$. We consider two cases:
\vspace{1em}
\\
\textbf{Case 1.} $v \notin Y_1$.
\vspace{0.5em}
\\
Let $X'_1=X_1 \cup \{v\}$. Let $S'_{e_1}$ be the double-star with pendant sets $X'_1$ and $Y_1$. Also, let $S'_{e}=S_{e}$, for $e \neq e_1$. Clearly, every double-star is $S_{k_1,k_2}$ or $S_{k_1+1,k_2}$.
\vspace{1em}
\\
\textbf{Case 2.} $v \in Y_1$.
\vspace{0.5em}
\\
Note that in this case, $v \notin X_2$. Since $|N_G(v_1) \cap N_G(v_2)| \leq k_1-1$, there exists $x \in X_2 \setminus X_1$. Let $X'_2=(X_2 \setminus \{x\}) \cup \{v\}$ and $Y'_1=(Y_1 \setminus \{v\}) \cup \{x\}$ and for $i \neq 2$ define $X'_i=X_i$ and for $j\geq 2$, $Y'_j=Y_j$.
\\
Now, $v \in X'_2$. If $v \notin Y'_2$, then we are done. If $v \in Y'_2$, then repeat this procedure till $v \in X'_j\setminus Y'_j$. Note that this can be done since $v \notin Y_t$. Let $S'_{e_i}$ be the double-star with pendant sets $X'_i$ and $Y'_i$ for $e_i \in E(C)$ and $S'_{e_j}=S_{e_j}$ for $e_j \notin E(C)$. Obviously, every double-star is $S_{k_1,k_2}$ or $S_{k_1+1,k_2}$.
\vspace{1em}
\\
For every $e=uv \in M$, there exists a unique double-star in $G \setminus M$ in which $u$ (or $v$) is a support vertex of degree $k_1$. These imply that $G$ is $\{S_{k_1,k_2}, S_{k_1+1,k_2}\}$-decomposable.
\end{proof}

Now, the following corollary is clear.
\begin{cor}
Let $G$ be a triangle-free $(2k+1)$-regular graph containing a perfect matching and $k_1, k_2$ be two positive integers such that $k_1 + k_2 = k-1$. Then $G$ is $\{S_{k_1, k_2},S_{k_1+1, k_2}\}$-decomposable. 
\end{cor}

We have an immediate corollary.
\begin{cor}
Let $r$ and $k$ be two positive integers such that $k\,|\,r$. Suppose that $S$ is a double-star of size $k$. If $G$ is a $2r$-regular graph, then $G$ is decomposed into $S$.
\end{cor}

\begin{proof}
Notice that since $k\,|\,r$, it is straight forward to see that $G$ can be decomposed into $G_1, \ldots, G_s$ such that each $G_i$ is spanning and $2k$-regular for $i=1, \ldots, s$. Now, Theorem \ref{double-n2n} yields that each $G_i$ is $S$-decomposable and this completes the proof.
\end{proof}

\begin{cor}
Let $G$ be a 1-factorable $(2k+1)$-regular graph. If $r \geq 3$ is a positive integer such that $r\,|\,2k+1$, then $G$ is $\{S_{k_1, k_2}, S_{k_1-1, k_2}\}$-decomposable, for any  double-star $S_{k_1, k_2}$ of size $r$.  
\end{cor}

\begin{proof}
Let $2k+1=rs$. It is easy to see that $G$ can be decomposed into spanning subgraphs $G_1, \ldots, G_s$ such that each $G_i$ is $r$-regular, for $i=1, \ldots, s$. Now, Theorem \ref{main} completes the proof.
\end{proof}

\end{document}